\documentclass[12pt]{article}
\usepackage[title]{appendix}
\usepackage{amsmath}    % need for subequations
\usepackage{amsfonts}
\usepackage{graphicx}   % need for figures
\usepackage{verbatim}   % useful for program listings
\usepackage{color}      % use if color is used in text
\usepackage{fullpage,mathpazo}
\usepackage{amssymb}
\usepackage{amsthm,amsmath}
\usepackage{bm}
\numberwithin{equation}{section}
\theoremstyle{definition}
\newtheorem{Def}{Definition}[section]
\newtheorem{Eg}{Example}[section]
\theoremstyle{plain}

\newtheorem{Lem}[Def]{Lemma}
\newtheorem{Thm}[Def]{Theorem}

\theoremstyle{definition}

\newcommand{\p}{\mathbb{P}}

\newcommand{\e}{\mathbb{E}}
\newcommand{\real}{\mathbb{R}}
\newcommand{\n}{\mathbb{N}}

\newcommand{\1}{{\bf 1}}
\newcommand{\mf}{\mathcal{F}}

\newcommand{\rd}{\mathrm{d}}

\newcommand{\eps}{\varepsilon}
\newcommand{\nt}{\widetilde{N}}
\newcommand{\pde}{\psi_{\delta,\varepsilon}}
\newcommand{\urde}{u_{r,\delta,\varepsilon}}

\allowdisplaybreaks

\begin{document}

\title{
	Existence of density functions for the Running Maximum of SDEs by non-truncated L\'evy processes.
}
\author{
	Takuya Nakagawa
	\footnote{
		Department of Mathematical Sciences,
		Ritsumeikan University,
		1-1-1 Nojihigashi, 
		Kusatsu, Shiga, 525-8577, 
		Japan. 
		Email: \texttt{takuya.nakagawa73@gmail.com}
	}
 \footnote{Corresponding Author}
 \quad and \quad
	Ryoichi Suzuki
	\footnote{
		Department of Mathematical Sciences,
		Ritsumeikan University,
		1-1-1 Nojihigashi, 
		Kusatsu, Shiga, 525-8577, 
		Japan. 
		Email: \texttt{rsuzukimath@gmail.com} 
  ORCiD: 0000-0001-9979-1882
	}
}
\maketitle
\begin{abstract}
We verify the existence of density functions of the running maximum of a stochastic differential equation (SDE) driven by a Brownian motion and a non-truncated stable process.
This is proved by the existence of density functions of the running maximum of Wiener-Poisson functionals resulting from Bismut’s approach to Malliavin calculus for jump processes. \\
%60H10  	Stochastic ordinary differential equations (aspects of stochastic analysis) 
%60G52  	Stable stochastic processes
%60H07  	Stochastic calculus of variations and the Malliavin calculus
\textbf{Keywords}:
Running maximum, Density functions, Malliavin calculus, Stochastic differential equation, L\'{e}vy processes.\\
\textbf{2020 Mathematics Subject Classification}: 60H10; 60G52; 60H07\\
\end{abstract}

\section{Introduction}
%In this article, 
We consider a solution of the following one-dimensional SDE
\begin{align}\label{SDE}
    \rd X_t=b(X_t) \rd t + \sigma_1 \rd W_t + \sigma_2 \rd L_t,\ X_0=x\in\real
\end{align}
for $t\ge 0$, where $\sigma_1$ and $\sigma_2$ are constants, $b:\real\to\real$ is once differentiable and its derivative is bounded, $W=\{W_t\}_{t\in[0,T]}$ is a standard Brownian motion and $L=\{L_t\}_{t\in[0,T]}$ is a L\'evy process with l\'evy triplet $(0,0,\nu)$.
The infinitesimal generator $A$ of $L$ is defined by
\begin{equation*}
A f(x):= \int_{\real \setminus \{0\}} \left\{f(x+y)-f(x)-1_{\{|y| \le 1 \}}yf'(x) \right\} \nu(\rd y) , %\textrm{~~for~any~~} 
\end{equation*}
for any $f\in C_b^2(\real)$ and $\ x\in\real.$
See, e.g., equation (3.18) in \cite{Ap09}.%ちょっとタイポ等を修正
The L\'evy measure $\nu$ satisfies assumptions \eqref{ass_nu1} and \eqref{ass_nu2}. 
We assume that $W$ and $L$ are independent.
In considering SDE \eqref{SDE}, we introduce the following SDE for each $n\in\n$:
\begin{align}\label{SDE_n}
    \rd X_t^{(n)}=b\left(X_t^{(n)}\right) \rd t + \sigma_1 \rd W_t + \sigma_2 \rd L_t^{(n)},\ X_0^{(n)}=x\in\real
\end{align}
for $t\in[0,T]$, where $L^{(n)} = \{L_t^{(n)}\}_{t\in[0,T]}$ is a truncated L\'{e}vy process with jump sizes larger than $n$ of $L$.
Let $X^\ast:=\{X_t^\ast\}_{t\in[0,T]}$ and $X^{(\ast,n)}:=\{X_t^{(\ast,n)}\}_{t\in[0,T]}$, defined as
\begin{align*}
    X^{\ast}_t =\sup_{t\in[0,t]} X_s,\ X_t^{(\ast,n)}=\sup_{s\in[0,t]}X_s^{(n)} \textrm{~for~each~} t\in [0,T] \textrm{~and~}n\in\n,
\end{align*}
be the running maximums of the solutions $X$ and $X^{(n)}$ to SDE \eqref{SDE} and \eqref{SDE_n}, respectively.
It is well-known that when $b$ is Lipschitz continuous, SDEs \eqref{SDE} and \eqref{SDE_n} have a unique solution (e.g. see \cite{Pr05}).
This paper aims to show that the distribution of $X_t^\ast$ is absolutely continuous on the Lebesgue measure on $\real$ for all $t>0$.
The running maximum process has received widespread attention in recent years as an interesting objective both practically and theoretically (cf. \cite{GoMiUr22,PiSc21}).
The following results for the special case of the law of $X^\ast$ are known.
The density function for the maximum of Brownian motion (i.e., $x=b=\sigma_2=0$ and $\sigma_1=1$) is well known. See, e.g., \cite{KaSh91}.
The law of the maximum of Levy motion (i.e., $x=b=\sigma_1=0$ and $\sigma_2=1$) is also well known. See, for example, \cite{Ch13,KuPa13}. 

The following prior studies are based on the simultaneous dealing of Brownian motion and truncation L\'evy processes.
Song and Zhang \cite{SoZh15} study the existence of distributional density of $X_t$ and the weak continuity in the first variable of the distributional density under full Hormander's conditions.
This proof is proved by showing the statement for $X_t^{(1)}$.
Song and Xie \cite{SoXi18} show the existence of density functions for the running maximum $X_t^{(\ast,1)}$ of a L\'evy It\^o diffusion.
They claimed that if $b$ is Lipschitz continuous in Lemma 4.3 of \cite{SoXi18}, they can prove the existence of the density function of $X_t^{(\ast,1)}$. However, we cannot follow them because the product of weakly convergent sequences does not necessarily converge to the product of their convergences.
These are proved similarly if jump size $n$ is a finite value.
However, to the best of our knowledge, the results of the non-truncated L\'evy process are not known.
This is because Bismut’s approach to Malliavin
calculus for jump processes can simply calculate the concrete form only for finite jumps using Proposition 2.11 of \cite{SoZh15}.
In this paper, we show the existence of a density function for $X^\ast$ using the proof method of \cite{SoXi18} and the fact that the Malliavin calculus for $L$ can be defined by the limit of that of $L^{(n)}$.

This paper is organized as follows:
Section \ref{Notaion} describes the notations used in this paper and the main theorem.
In Section \ref{Bismut}, we recall Bismut’s approach to the Malliavin calculus with jumps.
In Section \ref{Regularity}, we explain and extend the results of Song and Xie \cite{SoXi18}.
In Section \ref{Applying}, we apply the results of the previous section to our stochastic differential equations.
In Section \ref{proofmainTh1}, we prove Theorem \ref{mainThm1}, which is the main result of this paper.
Section \ref{Appendices} presents some lemmas needed for proving the main results.

\section{Notation and result}\label{Notaion}
%For any $n\in\n$, we denote $\Gamma_0^n:=\{z\in\real:\ 0<|z|<n\}$
Let $L = \{L_t\}_{0\le t\le T}$ and $L^{(n)} = \{L_t^{(n)}\}_{t\in[0,T]}$ be a pure-jump process with and one truncated by $[-n,n]\setminus\{0\}$, respectively.
The jump size of $L$ and $L^{(n)}$ at time $t$ is defined by $\Delta L_t = L_t-L_{t-}$ and $\Delta L_t^{(n)}:=L_t^{(n)}-L_{t-}^{(n)}$ for any $t>0$ and $\Delta L_0:= 0$ and $\Delta L_0^{(n)}:= 0$.
The Poisson random measure associated to $L$ and $L^{(n)}$ on $\mathcal{B}([0,T])\times \mathcal{B}(\real\setminus\{0\})$ is denoted by
$N(t,F)= \sum_{0\le s\le t} 1_F(\Delta L_s)$ and $N^{(n)}(t,G)= \sum_{0\le s\le t} 1_F(\Delta L_s^{(n)})$ for $t\in[0,T]$ and $F\in\mathcal{B}(\real\setminus\{0\})$, respectively.
The L\'evy measure of $L$ and $L^{(n)}$ on $\mathcal{B}(\real\setminus\{0\})$ is defined as $\nu(\rd z)=c(z)\rd z$ and $c(z)\1_{\{|z|< n\}}(z)\rd z$, where the positive function $c$ satisfies that there exist some constant $\beta>1$ and $C>0$ such that
\begin{align}
    &\int_{\real\setminus\{0\}} (1\wedge |z|^2)\nu(\rd z)<\infty,\quad\lim_{n\to\infty} \int_{\real\setminus\{0\}} |z|^p\1_{\{|z|\ge n\}}(z)\nu(\rd z)=0\textrm{~for~any~}p\in(1,\beta),\label{ass_nu1}\\
    &\sup_{n\in\n}\left|\int_{1\le |z|< n} z\nu(\rd z)\right|<\infty\textrm{~~and~~}\int_{0+}\nu(\rd z)=\infty, \quad \left|\frac{c'(z)}{c(z)}\right|\le C \left(1\vee \frac{1}{|z|}\right)\textrm{~for~any~}z\neq 0.\label{ass_nu2}
\end{align}
The compensated Poisson random measure of $L$ and $L^{(n)}$ is defined as $\nt$ and $\nt^{(n)}$, respectively.

\begin{Eg}
When $L$ is a symmetric $\alpha$-stable process with $\alpha\in(1,2)$, then for any $z\neq 0$
\begin{align*}
    c(z)=\frac{c_\alpha}{|z|^{1+\alpha}},\textrm{~where~}c_\alpha=\pi^{-1}\Gamma(\alpha+1)\sin\left(\frac{\alpha\pi}{2}\right),
\end{align*}
so that a L\'evy measure $\nu$ satisfies assumptions \eqref{ass_nu1} and \eqref{ass_nu2} for any $p\in(1,\alpha)$.
\end{Eg}
Our results are described below.
\begin{Thm}\label{mainThm1}
    Assume that $b:\real\to\real$ is once differentiable and its derivative is bounded and that a L\'evy measure $\nu$ of $L$ satisfies \eqref{ass_nu1} and \eqref{ass_nu2}.
    Let $\{X_t\}_{t\in[0,T]}$ be the solution to Eq. \eqref{SDE}.
    If $\sigma_1^2+\sigma_2^2\neq 0$, then for any $T>0$ the law of $X_T^\ast$ is absolutely continuous with respect to the Lebesgue measure.
\end{Thm}

We prove this result in Section 6.
To prepare for that proof, we introduce Malliavin Calculus.

\section{Bismut’s approach to Malliavin Calculus with Jumps}\label{Bismut}
This section covers some basic facts about Bismut’s approach to Malliavin calculus for jump processes (cf. \cite{Bi83,SoZh15,SoXi18} etc.).

Let $\Gamma\subset \real^d$ be an open set containing the origin.
Let us define
\begin{align}\label{AS_Gamma}
    \Gamma_0:=\Gamma\setminus \{0\},\quad
    \varrho(z):= 1\vee \mathbf{d}\left(z,\Gamma_0^c\right)^{-1}
\end{align}
where $\mathbf{d}\left(z,\Gamma_0^c\right)^{-1}$ is the distance of $z$ to the complement of $\Gamma_0$.
Let $\Omega$ be the canonical space of all points $\omega=(w,\mu)$, where
\begin{itemize}
    \item $w:[0,1]\to \real^d$ is continuous function with $w(0)=0$;
    \item $\mu$ is an integer-valued measure on $[0,1]\times \Gamma_0$ with $\mu(A)<+\infty$ for any compact set $A\subset [0,1]\times \Gamma_0$.
\end{itemize}
Define the canonical process on $\Omega$ as follows: $\omega=(w,\mu)$,
\begin{align*}
    W_t(\omega):=w(t),\quad N(\omega; \rd t,\rd z):=\mu(\omega;\rd t ,\rd z):=\mu(\rd t,\rd z).
\end{align*}
Let $(\mathcal{F}_t)_{t\in[0,1]}$ be the smallest right-continuous filtration on $\Omega$ such that $W$ and $N$ are optional.
In the following, we write $\mathcal{F}:=\mathcal{F}_1$, and endow $(\Omega,\mathcal{F})$ with unique probability measure $\p$ such that
\begin{itemize}
    \item $W$ is a standard $d$-dimensional Brownian motion;
    \item $N$ is a Poisson random measure with intensity $\rd t\nu(\rd z)$, where $\nu(\rd z)=\kappa(z)\rd z$ with
    \begin{align}\label{AS_kappa}
        \kappa\in C^1(\Gamma_0; (0,\infty)),\ \int_{\Gamma_0}(1\wedge |z|^2)\kappa(z)\rd z<+\infty,\ |\nabla \log \kappa(z)|\le C \varrho (z),
    \end{align}
    where $\varrho(z)$ is defined by Eq \eqref{AS_Gamma};
    \item $W$ and $N$ are independent.
\end{itemize}
%\textcolor{blue}{When $\Gamma_0=(-n,n)\setminus \{0\}$, $\mathbf{d}(z,\Gamma_0^c)=(|z-n|\wedge |z+n|)\1_{\Gamma_0}(z)$ and $\varrho(z)=\begin{cases}
%    \infty & z=0\\
%    0 & z\notin [-n,n]\setminus\{0\}\\
%    1 & |z-n|\wedge |z+n|>1\\
%    (|z-n|\wedge |z+n|)^{-1} & \textrm{otherwise}
%\end{cases}$. $|\nabla \log \kappa(z)|=\left|\nabla \log \frac{c(z)}{|z|^{1+\alpha}}\1_{\{0<|z|<n\}}(z)\right|=\left|\frac{c'(z)}{c(z)}-(1+\alpha)\frac{1}{|z|}\right|\1_{\{0<|z|<n\}}(z)\le C \varrho(z)$.\\ It does not matter if $c$ and $c'$ are bounded on any compact set on $\real$.}\\
In the following, we write the compensated Poisson random measure  $N$ as by
\begin{align*}
    \widetilde{N}(\rd t,\rd z):=N(\rd t,\rd z)-\nu(\rd z)\rd t.
\end{align*}
Let $p\ge 1$ and $i=1,2$.% and $k$ be an integer.
We introduce the following spaces for later use.
\begin{itemize}
    \item $L^p(\Omega)$:
        The space of all $\mathcal{F}$-measurable random variables with finite norm:
    \begin{align*}
        \|F\|_p:=\e\left[|F|^p\right]^{\frac{1}{p}}.
    \end{align*}
    \item $\mathbb{L}_p^i$:
        The space of all predictable processes:
            $\xi:\Omega\times [0,1]\times\Gamma_0\to \real^k$ with finite norm:
    \begin{align*}
        \|\xi\|_{\mathbb{L}_p^i}
        :=\e\left[\left(\int_0^1 \int_{\Gamma_0}|\xi(s,z)|^i \nu(\rd z)\rd s\right)^{\frac{p}{i}} \right]^{\frac{1}{p}}
            +\e\left[\int_0^1 \int_{\Gamma_0}|\xi(s,z)|^p \nu(\rd z)\rd s \right]^{\frac{1}{p}}
            <\infty.
    \end{align*}
    \item $\mathbb{H}_p$: The space of all measurable adapted processes $h:\Omega\times [0,1]\to\real^d$ with finite norm:
        \begin{align*}
            \|h\|_{\mathbb{H}_p}
            :=\e\left[\left(\int_0^1 |h(s)|^2 \rd s\right)^{\frac{p}{2}} \right]^{\frac{1}{p}}
            <\infty.
        \end{align*}
    \item $\mathbb{V}_p$: The space of all predictable processes $\mathbf{v}: \Omega\times [0,1]\times \Gamma_0\to\real^d$ with finite norm:
    \begin{align*}
        \|\mathbf{v}\|_{\mathbb{V}_p}
        :=\|\nabla_z \mathbf{v}\|_{\mathbb{L}_p^1}
            +\|\mathbf{v}\varrho\|_{\mathbb{L}_p^1}
            <\infty,
    \end{align*}
    where $\varrho(z)$ is defined by Eq \eqref{AS_Gamma}.
    Hereafter denoted as
    \begin{align*}
        \mathbb{H}_{\infty-}:=\bigcap_{p\ge 1}\mathbb{H}_p,\quad
        \mathbb{V}_{\infty-}:=\bigcap_{p\ge 1}\mathbb{V}_p.
    \end{align*}
    \item $\mathbb{H}_0$: The space of all bounded measurable adapted processes $h:\Omega\times [0,1]\to \real^d$.
    \item $\mathbb{V}_0$: The space of all predictable processes $\mathbf{v}:\Omega\times [0,1]\times \Gamma_0\to \real^d$ with the following properties:
    (i) $\mathbf{v}$ and $\nabla_z \mathbf{v}$ are bounded;
    (ii) there exists a compact subset $U\subset \Gamma_0$ such that
    \begin{align*}
        \mathbf{v}(t,z)=0,\quad \forall z\in U.
    \end{align*}
\end{itemize}
Let $C_p^\infty (\real^m)$ be the class of all smooth functions on $\real^m$ whose all of their derivatives have at most polynomial growth.
Let $\mathcal{F}C_p$ be the class of all Wiener-Poisson functionals on $\Omega$ with the following form:
\begin{align*}
    F(\omega)=f(W(h_1),\ldots,W(h_{m_1}),N(g_1),\ldots,N(g_{m_2})),
\end{align*}
where $f\in C_p^\infty (\real^{m_1+m_2}),h_1,\ldots,h_{m_1}\in\mathbb{H}_0$ and $g_1,\ldots,g_{m_2}\in\mathbb{V}_0$ are non-random and real-valued, and for $j=1,2,\ldots,m_1$ and $k=1,2,\ldots, m_2$, we define
\begin{align*}
    W(h_j):=\int_0^1 \langle h_j(s),\rd W_s \rangle_{\real^d}\textrm{~and~}
    N(g_k):=\int_0^1 \int_{\Gamma_0} g_k(s,z) N(\rd s, \rd z).
\end{align*}
For any $p>1$ and $\Theta=(h,\mathbf{v})\in\mathbb{H}_p\times \mathbb{V}_p$, let us denote
\begin{align*}
    D_{\Theta}F
    :=\sum_{i=1}^{m_1} (\partial_i f)(\cdot)\int_0^1 \langle h(s),h_i \rangle_{\real^d} \rd s
    +\sum_{j=1}^{m_2} (\partial_{j+m_1} f)(\cdot) \int_0^1 \int_{\Gamma_0} \langle \mathbf{v}(s,z),\nabla_z g_j \rangle_{\real^d}  N(\rd s, \rd z),
\end{align*}
where “$(\cdot)$" stands for $W(h_1),\ldots,W(h_{m_1}),N(g_1),\ldots,N(g_{m_2})$.
\begin{Def}
    For $p>1$ and $\Theta=(h,\mathbf{v})\in\mathbb{H}_p\times \mathbb{V}_p$, we define the first order Sobolev space $\mathbb{W}_\Theta^{1,p}$ being the completion of $\mathcal{F}C_p$ in $L^p(\Omega)$ with respect to the norm:
    \begin{align*}
        \|F\|_{\Theta;1,p}:=\|F\|_{L^p}+\|D_\Theta F\|_{L^p}.
    \end{align*}
\end{Def}
The Banach space $\mathbb{W}_\Theta^{1,p}$ has weak compactness, which is the key to the proof of Theorem \ref{mainThm1}. (see Lemma 2.3 in \cite{SoXi18}).
Next, we give the results of applying the Malliavin calculus we just set up to Running Maximum Processes.

\section{Regularity of Running Maximum Processes}\label{Regularity}
In this section, we discuss the results of Song and Xie \cite{SoXi18} and their extensions.
Let $X^{(n)} = \{X_s^{(n)}\}_{s\ge 0}$ be a right continuous real-valued process.
For any fixed $T>0$ and $n\in\n$, in the following we shall write
\begin{align*}
    X_T^{(\ast,n)}:=\sup_{s\in[0,T]}X_s^{(n)}.
\end{align*}
In the same way as for Proposition 3.1 in \cite{SoXi18}, the following holds for each $n$.
\begin{Lem}\label{SoXiePro31}(\cite{SoXi18}, Proposition 3.1)\\
Let $X^{(n)}=\left\{X_s^{(n)}\right\}_{s\ge 0}$ be a right continuous process.
Suppose that for some $p>1$ and $\Theta= (h,\mathbf{v})\in \mathbb{H}_{\infty-}\times \mathbb{V}_{\infty-}$,
\begin{itemize}
    \item [1.] $\e\left[\left|X_s^{(\ast,n)}\right|^p \right]<\infty$, and for any $s\in[0,T]$, $X_s^{(n)}\in \mathbb{W}_{\Theta}^{1,p}$ and 
    \begin{align*}
        \e\left[\sup_{s\in[0,T]}\left|D_\Theta X_s^{(n)}\right|^p\right]<\infty,
    \end{align*}
    \item [2.] the process $\left\{D_\Theta X_s^{(n)}\right\}_{s\in[0,T]}$ processes a right continuous version.
\end{itemize}
Then $X_T^{(\ast,n)}\in \mathbb{W}_\Theta^{1,p}$ and 
\begin{align*}
    D_\Theta X_T^{(\ast,n)}\le \left(D_\Theta X^{(n)}\right)_T^\ast :=\sup_{s\in[0,T]} D_\Theta X_s^{(n)}.
\end{align*}
\end{Lem}
The same holds true for the limit with respect to $n$.
\begin{Lem}\label{SoXieThm32}
In the setup of Lemma \ref{SoXiePro31}.
We set
\begin{align*}
    X_T^\ast :=\sup_{s\in[0,T]} X_s.
\end{align*}
In addition, for some $p>1$
\begin{align*}
    \lim_{n\to\infty}\e\left[\sup_{t\in[0,T]}\left|X_t^{(n)}-X_t\right|^p\right]=0.
\end{align*}
Then $X_T^\ast\in \mathbb{W}_\theta^{1,p}$ and a sequence $\{D_\Theta X_s^{(n)}\}_{s\in[0,T]}$ converges to $D_\Theta X=\{D_\Theta X_s\}_{s\in[0,T]}$ in the weak topology of $L^p(\Omega\times [0,T])$.
Moreover, suppose that this $D_\Theta X$ satisfies the following assumptions:
\begin{itemize}
    \item [1.] $\e\left[\left|X_s^{\ast}\right|^p \right]<\infty$, and for any $s\in[0,T]$, $X_s\in \mathbb{W}_{\Theta}^{1,p}$ and 
    \begin{align*}
        \e\left[\sup_{s\in[0,T]}\left|D_\Theta X_s\right|^p\right]<\infty,
    \end{align*}
    \item [2.] the process $\left\{D_\Theta X_s\right\}_{s\ge 0}$ processes a right continuous version.
\end{itemize}
Then, 
    \begin{align*}
        D_\Theta X_T^\ast \le \left(D_\Theta X\right)_T^\ast := \sup_{s\in[0,T]} D_\Theta X_s,
    \end{align*}
    and if
\begin{align*}
    \p\left(D_\Theta X_t\neq 0\textrm{~on~} \left\{t\in (0,T]: X_t=X_t^{\ast}\right\}\right)=1,
\end{align*}
then the law of $X_T^{\ast}$ is absolutely continuous with respect to the Lebesgue measure.
\end{Lem}
\begin{proof}
    From Lemma 2.3 in \cite{SoXi18} and assumptions, we obtain $X_T^\ast\in \mathbb{W}_\theta^{1,p}$ and
    \begin{align*}
        \lim_{n\to\infty} D_\Theta X_\cdot^{(n)} =D_\Theta X_\cdot \textrm{~~weakly~in~}L^p(\Omega\times[0,T]).
    \end{align*}
    This can be proved in the same way as in Proposition 3.1 and Theorem 3.2 in \cite{SoXi18} since
    \begin{align*}
        1&=\p\left(D_\Theta X_t=D_\Theta X_T^\ast\textrm{~on~} \left\{t\in [0,T]: X_t=X_t^{\ast}\right\}\right)\\
        &\le \p\left(D_\Theta X_t=D_\Theta X_T^\ast\textrm{~on~} \left\{t\in (0,T]: X_t=X_t^{\ast}\right\}\right)\\
        &=1.
    \end{align*}
    In addition, by the closability of $D_\Theta$ (see Theorem 2.6 in \cite{SoXi18}), we obtain
    \begin{align*}
        \p\left(\1_{A}\left(D_\Theta X_T^\ast\right) D_\Theta X_T^\ast =0\right)=1,
    \end{align*}
    for any $A\in\mathcal{B}(\real)$ with $\textrm{Leb}(A)=0$.
    Thus, we have
    \begin{align*}
        1&=\p(\left\{\1_{A}\left(D_\Theta X_T^\ast\right) D_\Theta X_T^\ast =0\right\} 
        \cap
        \left\{D_\Theta X_t=D_\Theta X_T^\ast\textrm{~on~} \left\{t\in (0,T]: X_t=X_t^{\ast}\right\}\right\}\\
        &\quad\quad\cap
        \left\{D_\Theta X_t\neq 0\textrm{~on~} \left\{t\in (0,T]: X_t=X_t^{\ast}\right\}\right\}
        )\\
        &=\p(\left\{\1_{A}\left(D_\Theta X_T^\ast\right) D_\Theta X_t =0\textrm{~on~} \left\{t\in (0,T]: X_t=X_t^{\ast}\right\}\right\}\\
        &\quad\quad\cap
        \left\{D_\Theta X_t=D_\Theta X_T^\ast\textrm{~on~} \left\{t\in (0,T]: X_t=X_t^{\ast}\right\}\right\}\\
        &\quad\quad\cap
        \left\{D_\Theta X_t\neq 0\textrm{~on~} \left\{t\in (0,T]: X_t=X_t^{\ast}\right\}\right\}
        )\\
        &=\p(\left\{\1_{A}\left(D_\Theta X_T^\ast\right) =0\right\} 
        \cap
        \left\{D_\Theta X_t=D_\Theta X_T^\ast\textrm{~on~} \left\{t\in (0,T]: X_t=X_t^{\ast}\right\}\right\}\\
        &\quad\quad\cap
        \left\{D_\Theta X_t\neq 0\textrm{~on~} \left\{t\in (0,T]: X_t=X_t^{\ast}\right\}\right\}
        )\\
        &\le\p(\1_{A}\left(D_\Theta X_T^\ast\right) =0)\\
        &=1.
    \end{align*}
\end{proof}
Now we know the relationship between the Malliavin calculus of running maximum processes and the existence of the density function.
Next, we note the results of applying the Malliavin calculus to the SDE \eqref{SDE}.
\section{Applying Malliavin calculus to SDEs}\label{Applying}
In this section, to find the equation satisfied by $D_\Theta X$ for $X$ in equation \eqref{SDE}, we check an equation satisfied by $D_\Theta X^{(n)}$ for $X^{(n)}$ for equation \eqref{SDE_n}.
This is shown in the same way as for Lemma 4.3 in \cite{SoXi18}.
\begin{Lem}(\cite{SoXi18}, Lemma 4.3)\label{Lem4.3}\\
Assume that $b:\real\to \real$ is once differentiable and its derivative is bounded.
Then for any $\Theta= (h,\mathbf{v})\in \mathbb{H}_{\infty-}\times \mathbb{V}_{\infty-}$ and $t\in [0,T]$, $X_t^{(n)}\in \mathbb{W}_\Theta^{1,2}$ and
\begin{align*}
    D_\Theta X_t^{(n)}=\int_0^t b'\left(X_s^{(n)}\right) D_\Theta X_s^{(n)} \rd s + \sigma_1 \int_0^t h(s)\rd s + \sigma_2 \int_0^t \int_{0<|z|\le n} \mathbf{v}(s,z)N(\rd s,\rd z).
\end{align*}
\end{Lem}
\begin{proof}
\begin{align*}
    X_t^{(n)}&=x+\int_0^t b\left(X_s^{(n)}\right)\rd s+\sigma_1 W_t + \sigma_2 L_t^{(n)}\\
    &=x+\int_0^t b\left(X_s^{(n)}\right)\rd s+\sigma_1 W_t\\ 
        &\qquad + \sigma_2 \left(\int_0^t \int_{|z|\ge 1} z \1_{\{0<|z|< n\}} N(\rd s,\rd z)+ \int_0^t \int_{0<|z|< 1} z \1_{\{0<|z|< n\}} \nt (\rd s,\rd z)\right)\\
    &=x+\int_0^t b\left(X_s^{(n)}\right)\rd s+\sigma_1 W_t + \sigma_2 \left(\int_0^t \int_{1\le |z|< n} z \rd s \nu(\rd z)+ \int_0^t \int_{0<|z|< n} z \nt (\rd s,\rd z)\right)\\
    &=x+\int_0^t b_n\left(X_s^{(n)}\right)\rd s+\sigma_1 W_t + \sigma_2 \int_0^t \int_{0<|z|< n} z \nt (\rd s,\rd z)
\end{align*}
where 
\begin{align*}
b_n(x):=b(x)+\int_{1\le |z|< n} z \nu (\rd z).
\end{align*}
Obviously, $b_n$ is once differentiable, and its derivative is bounded for each $n\in\n$ and $b_n'=b'$.
We will calculate $D_\Theta X_t$.
For each $m,n\in\n$, consider SDE
\begin{align*}
    X_t^{(m,n)}=x+\int_0^t b_{n}\left(X_t^{(m,n)}\right)+\sigma_1 W_t+ \sigma_2 \int_0^t \int_{0<|z|< n} z \nt (\rd s,\rd z).
\end{align*}
Then we have by triangle inequality, Gronwall's inequality, and Kunita's first inequality (\cite{Ap09} Theorem 4.4.23)
\begin{align*}
    \e\left[\sup_{s\le t}\left|X_s^{(m,n)}\right|^2\right]\le 4 |x|^2+4 t \e\left[\int_0^t b_{n}^2 \left(X_s^{(m,n)}\right)\rd s\right]\rd s + 4\sigma_1^2 t + 4\sigma_2^2 t \int_{0< |z|< n}|z|^2 \nu(\rd z).
\end{align*}
    Subsequent proofs can be done in precisely the same way as in Lemma 4.3 in \cite{SoXi18}.
\end{proof}
The following lemma defines $D_\Theta X$ and confirms that it satisfies the following equation.
This $D_\Theta X$ is defined in the limit of weak $L^p(\Omega\times[0,T])$ convergence of $D_\Theta X^{(n)}$, but it is found to be strongly $L^p(\Omega\times [0,T])$ convergent in practice.
\begin{Lem}\label{DthetaXt}
    Assume the same assumptions as in Lemma \ref{Lem4.3}.
Then for some $p\in(1,\beta)$, for some $n\in\n$, for any $q\in(1,\beta)$ and for any $\Theta= (h,\mathbf{v})\in \mathbb{H}_{\infty-}\times \mathbb{V}_{\infty-}$, where
\begin{align*}
        \lim_{n\to\infty}\int_0^T \int_{|z|>n} \left|\mathbf{v}(s,z)\right|^{\frac{q}{\beta}} \rd s \nu(\rd z)=0 \textrm{~~and~~}
        \int_0^T \int_{|z|>n} \left|\mathbf{v}(s,z)\right|^q \rd s \nu(\rd z)<\infty,
    \end{align*}
     $X_t^\ast\in \mathbb{W}_\Theta^{1,p}$ for any $t\in [0,T]$ and
\begin{align}\label{DX_eq}
    D_\Theta X_t=\int_0^t b'\left(X_s\right) D_\Theta X_s \rd s + \sigma_1 \int_0^t h(s)\rd s + \sigma_2 \int_0^t \int_{|z|>0} \mathbf{v}(s,z)N(\rd s,\rd z).
\end{align}
\end{Lem}
\begin{proof}
    By using Lemma 2.3 in \cite{SoXi18}, Lemma \ref{XnX_Lp_conv} and the closability of $D_\Theta$ (cf. Lemma 2.7 in [\cite{SoZh15}), we have
    \begin{align*}
        \lim_{n\to\infty}D_\Theta X_\cdot^{(n)} = D_\Theta X_\cdot \textrm{~~weakly~~in~~}L^p(\Omega\times [0,T]).
    \end{align*}
We verify that this $D_\Theta X=\{D_\Theta X_t\}_{t\in[0,T]}$ satisfies equation \eqref{DX_eq}.
We set $\{Y_t\}_{t\in [0,T]}$ as a solution of
\begin{align*}
    Y_t&=\int_0^t b'(X_s) Y_s \rd s + C_t,\textrm{~~where}\\
    C_t&=\sigma_1 \int_0^t h(s)\rd s + \sigma_2 \int_0^t \int_{|z|>0} \mathbf{v}(s,z)N(\rd s,\rd z),\\
    C_t^{(n)}&=\sigma_1 \int_0^t h(s)\rd s + \sigma_2 \int_0^t \int_{0<|z|<n} \mathbf{v}(s,z)N(\rd s,\rd z).
\end{align*}
We prove
\begin{align}\label{DXnY_Lp_conv}
    \lim_{n\to\infty}\e\left[\sup_{t\in[0,T]}\left|D_\Theta X_t^{(n)}-Y_t\right|^p\right]=0.
\end{align}
By using inequality $|a+b|^p \le 2^{p-1}(|a|^p+|b|^p)$ for any $a,b\in\real$, we obtain
\begin{align*}
    &\e\left[\sup_{t\in[0,T]}\left|D_\Theta X_t^{(n)}-Y_t\right|^p\right]\\
    &\le 2^{p-1}\e\left[\sup_{t\in[0,T]}\left| \int_0^t \left\{b'\left(X_s^{(n)}\right) D_\Theta X_s^{(n)}-b'\left(X_s\right) Y_s\right\}\rd s \right|^p\right]
    +2^{p-1} \e\left[\sup_{t\in[0,T]}\left|C_t^{(n)}-C_t \right|^p \right]\\
    %&\le 2^{2(p-1)} \e\left[\sup_{t\in[0,T]}\left| \int_0^t \left\{b'\left(X_s^{(n)}\right)-b'\left(X_s\right)\right\} D_\Theta X_s^{(n)}\rd s \right|^p\right]\\
    %    &\quad+
    %    2^{2(p-1)}\e\left[\sup_{t\in[0,T]}\left| \int_0^t b'\left(X_s\right)\left\{ D_\Theta X_s^{(n)}- Y_s\right\}\rd s \right|^p\right]\\
    %    &\quad+
    %    2^{p-1} \e\left[\sup_{t\in[0,T]}\left|C_t^{(n)}-C_t \right|^p \right]\\
    %&\le 2^{2(p-1)} \e\left[\sup_{t\in[0,T]} \int_0^t \left|b'\left(X_s^{(n)}\right)-b'\left(X_s\right)\right|^p \left|D_\Theta X_s^{(n)}\right|^p\rd s \right]
     %   +
     %   2^{2(p-1)}\e\left[\sup_{t\in[0,T]} \int_0^t \left|b'\left(X_s\right)\right|^p\left| D_\Theta X_s^{(n)}- Y_s\right|^p\rd s \right]\\
     %   &\quad+
     %   2^{p-1} \e\left[\sup_{t\in[0,T]}\left|C_t^{(n)}-C_t \right|^p \right]\\
    &\le 2^{2(p-1)} \int_0^T \e\left[\left|b'\left(X_s^{(n)}\right)-b'\left(X_s\right)\right|^p \left|D_\Theta X_s^{(n)}\right|^p\right]\rd s\\ 
        &\quad+
        2^{2(p-1)}\left\|b'\right\|^p_\infty \int_0^T \e\left[\left| D_\Theta X_s^{(n)}- Y_s\right|^p\right]\rd s
        +
        2^{p-1} \e\left[\sup_{t\in[0,T]}\left|C_t^{(n)}-C_t \right|^p \right].
\end{align*}
The last and second inequality from the last follows from Jensen's inequality and Fubini's theorem.
By using Gronwall's inequality, we have
\begin{align*}
    &\e\left[\sup_{t\in[0,T]}\left|D_\Theta X_t^{(n)}-Y_t\right|^p\right]\\
    &\le 2^{2(p-1)}\exp\left(2^{2(p-1)}\|b'\|_\infty^p\right)\int_0^T \e\left[\left|b'\left(X_s^{(n)}\right)-b'\left(X_s\right)\right|^p \left|D_\Theta X_s^{(n)}\right|^p\right]\rd s\\
    &\quad+2^{2(p-1)}\exp\left(2^{2(p-1)}\|b'\|_\infty^p\right)\e\left[\sup_{t\in[0,T]}\left|C_t^{(n)}-C_t \right|^p \right]
    .
\end{align*}
We show
\begin{align}
    \lim_{n\to\infty}\int_0^T \e\left[\left|b'\left(X_s^{(n)}\right)-b'\left(X_s\right)\right|^p \left|D_\Theta X_s^{(n)}\right|^p\right]\rd s&=0 \label{bXnbX_Lp_conv},\\
    \lim_{n\to\infty}\e\left[\sup_{t\in[0,T]}\left|C_t^{(n)}-C_t \right|^p \right]&=0. \label{CnC_Lp_conv}
\end{align}
See Lemma \eqref{lem_CnC_Lp_conv} for proof of \eqref{CnC_Lp_conv}.
Here we show an equation \eqref{bXnbX_Lp_conv}.
Notice that $p\in(1,\beta)$, there exists $q>1$ such that $pq<\beta$ because of density theorem.
By using H\"older inequality, we have
\begin{align*}
    &\int_0^T \e\left[\left|b'\left(X_s^{(n)}\right)-b'\left(X_s\right)\right|^p \left|D_\Theta X_s^{(n)}\right|^p\right]\rd s\\
    &\le 
        \int_0^T \e\left[\left|b'\left(X_s^{(n)}\right)-b'\left(X_s\right)\right|^{\frac{pq}{q-1}}\right]^{\frac{q-1}{q}} \e\left[ \left|D_\Theta X_s^{(n)}\right|^{pq}\right]^{\frac{1}{q}}\rd s\\
    &\le
        T \e\left[\sup_{t\in[0,T]}\left|b'\left(X_t^{(n)}\right)-b'\left(X_t\right)\right|^{\frac{pq}{q-1}}\right]^{\frac{q-1}{q}} \e\left[ \sup_{t\in[0,T]}\left|D_\Theta X_t^{(n)}\right|^{pq}\right]^{\frac{1}{q}}.
\end{align*}
Since an inequality $|a+b|^p \le 2^{p-1}(|a|^p+|b|^p)$ for any $a,b\in\real$ and $p\ge 1$ and Jensen's inequality, we have
\begin{align*}
   &\e\left[\sup_{t\in[0,T]}\left|D_\Theta X_t^{(n)}\right|^{pq}\right]\\
    &\le
     2^{pq-1}\e\left[\sup_{t\in[0,T]}\left|\int_0^t b'\left(X_s^{(n)}\right) D_\Theta X_s^{(n)} \rd s\right|^{pq}\right]
     +2^{2(pq-1)}\e\left[\sup_{t\in[0,T]}\left|\int_0^t h(s)\rd s \right|^{pq}\right]\\
     &\qquad
     +2^{2(pq-1)}\e\left[\sup_{t\in[0,T]}\left|\int_0^t \int_{0<|z|< n} \mathbf{v}(s,z) N(\rd s,\rd z)\right|^{pq}\right]\\
    &\le 2^{pq-1}\|b'\|_\infty^{pq} \int_0^T \e\left[\sup_{u\in[0,s]}\left| D_\Theta X_u^{(n)}\right|^{pq}\right] \rd s\\
    &\qquad +2^{2(pq-1)}\left(\e\left[\int_0^T \left| h(s)\right|^{pq}\rd s \right]
     +\e\left[\sup_{t\in[0,T]}\left|\int_0^t \int_{0<|z|< n} \mathbf{v}(s,z) N(\rd s,\rd z)\right|^{pq}\right]\right).
\end{align*}
Gronwall’s inequality implies
\begin{align*}
    &\e\left[\sup_{t\in[0,T]}\left|D_\Theta X_t^{(n)}\right|^{pq}\right]\\
    &\le
    2^{2(pq-1)}\left(\e\left[\int_0^T \left| h(s)\right|^{pq}\rd s \right]
     +\e\left[\sup_{t\in[0,T]}\left|\int_0^t \int_{0<|z|< n} \mathbf{v}(s,z) N(\rd s,\rd z)\right|^{pq}\right]\right)
     e^{2^{pq-1}T \|b'\|_\infty^{pq}}.
\end{align*}
The boundedness of sup means with respect to time can be proved as in Lemma \eqref{XnX_Lp_conv} (ii).
Since assumptions of $h$ and $\mathbf{v}$, we have
\begin{align*}
    \sup_{n\in\n}\e\left[\sup_{t\in[0,T]}\left|D_\Theta X_t^{(n)}\right|^{pq}\right]<\infty.
\end{align*}
By Lemma \ref{XnX_Lp_conv}, boundedness of $b'$ and continuous mapping theorem, we have
\begin{align*}
    \lim_{n\to\infty}\e\left[\sup_{t\in[0,T]}\left|b'\left(X_t^{(n)}\right)-b'\left(X_t\right)\right|^{\frac{pq}{q-1}}\right]^{\frac{q-1}{q}}=0.
\end{align*}
We have \eqref{DXnY_Lp_conv}.
Since there is only one weak convergence destination, \eqref{DX_eq} follows.
\end{proof}
The proof of Theorem \ref{mainThm1} is now ready to be presented.
\section{Proof of Theorem \ref{mainThm1}}\label{proofmainTh1}
\begin{proof}
Applying It\^o formula to $e^{-\int_0^t b'(X_s) \rd s}D_\Theta X_t$ (e.g., see \cite{Pr05}, Corollary (Integration by Parts), P.84), we obtain
\begin{align*}
    e^{-\int_0^t b'\left(X_s \right) \rd s}D_\Theta X_t &= \int_{0+}^t e^{-\int_0^s b'\left(X_u \right) \rd u}\circ \rd D_\Theta X_{s-}
        +\int_{0+}^t D_\Theta X_{s-} \circ \rd e^{-\int_0^s b'\left(X_u \right) \rd u}\\
    &=\int_{0}^t e^{-\int_0^s b'\left(X_u \right) \rd u} \rd D_\Theta X_{s-}
        +\int_{0}^t D_\Theta X_{s-} \rd e^{-\int_0^s b'\left(X_u \right) \rd u}\\
    &=\int_{0}^t e^{-\int_0^s b'\left(X_u \right) \rd u} \left(b'\left(X_s \right) D_\Theta X_s+\sigma_1 h(s)\right)\rd s\\
    &\qquad+\int_{0}^t e^{-\int_0^s b'\left(X_u \right) \rd u} \sigma_2 \int_{|z|>0} \mathbf{v}(s,z) N(\rd s,\rd z)\\
    &\qquad +\int_0^t D_\Theta X_{s-}\left(-b'\left(X_s \right) \int_0^t e^{-\int_0^s b'\left(X_u \right) \rd u}\right)\rd s.
\end{align*}
For any $t>0$, set
\begin{align}\label{eta_function}
    h(t):=\sigma_1 e^{-\int_0^t b'\left(X_s \right) \rd s},\quad
    \mathbf{v}(t,z):=\sigma_2 e^{-\int_0^t b'\left(X_s \right) \rd s} \eta(z),\quad
    \eta(z)=\begin{cases}
        |z|^2, & |z|\le \frac{1}{4}\\
        0, & |z|>\frac{1}{2}\\
        \textrm{smooth}, & \textrm{otherwise}.
    \end{cases}
\end{align}
Since the function $\mathbf{v}$ is bounded, it satisfies the assumptions of Lemma \ref{DthetaXt}.
We have
\begin{align*}
    D_\Theta X_t&=e^{\int_0^t b'\left(X_s \right) \rd s}\left(\sigma_1^2 \int_0^t e^{-2\int_0^s b'\left(X_u \right) \rd u}\rd s +\sigma_2^2 \int_0^t \int_{|z|>0} e^{-2\int_0^s b'\left(X_u \right) \rd u}\eta(z) N(\rd s,\rd z)\right)\\
    & \ge e^{\int_0^t \left(b'\left(X_s \right)-2\|b\|_{\textrm{Lip}}\right)\rd s} \left(\sigma_1^2 t +\sigma_2^2 \int_0^t \int_{|z|>0} \eta(z) N(\rd s,\rd z)\right).
\end{align*}
Noticing the condition $\sigma_1^2+\sigma_2^2\neq 0$ and the fact
\begin{align}\label{prob1}
    \p\left(\int_0^t \int_{|z|>0} \eta(z) N(\rd s,\rd z),\ \forall t>0\right)=1,
\end{align}
we have
\begin{align*}
    \p\left(D_\Theta X_t>0,\ \forall t\in(0,T]\right)=1.
\end{align*}
See the section \ref{prob1proof} for proof of Eq. \eqref{prob1}.
So we have
\begin{align*}
    1=\p\left(D_\Theta X_t>0,\ \forall t\in(0,T]\right)\le \p\left(D_\Theta X_t\neq 0\textrm{~on~} \left\{t\in (0,T]: X_t=\sup_{s\in [0,t]}X_s\right\}\right)=1.
\end{align*}
Therefore, we conclude by Lemma \ref{SoXieThm32} that the law of $X_T^\ast$ is absolutely continuous with respect to the Lebesgue measure.
\end{proof}

\section*{Declarations}\label{Declarations}
{\bf Conflict of interest}\quad
The authors have no relevant financial or non-financial interests to disclose.
\section*{Funding}\label{Funding}
The second author was supported by JSPS KAKENHI Grant Number 23K12507.

\begin{appendices}
\section{Appendices}\label{Appendices}
We give some lemma to show Theorem \ref{mainThm1}.
In subsection \ref{prob1proof}, we confirm an equation \eqref{prob1}.
In subsection \ref{ProofofLpconv} we prove that $\{X^{(n)}\}_{n\in\n}$ converges to $X$ and in subsection \ref{Preparation} we prepare for it.

\subsection{A proof of Eq. \eqref{prob1}}\label{prob1proof}
In this section, we prove Eq. \eqref{prob1}. 
We set $t>0$, $\eta$ as \eqref{eta_function} and $\eps_k=\frac{1}{2^k}$ for any $k\in\n$.
Since
\begin{align*}
    \lim_{k\to\infty} \int_0^t \int_{\eps_k <|z|< n} \eta(z) N(\rd s,\rd z)&=\int_0^t \int_{0 <|z|\le n} \eta(z) N(\rd s,\rd z)\textrm{~~in~~} L^2(\Omega),\\
    \lim_{k\to\infty} \int_0^t \int_{\eps_k <|z|< n} \eta(z) N(\rd s,\rd z)&=\int_0^t \int_{0 <|z|< n} \eta(z) N(\rd s,\rd z) \textrm{~~in~distribution.}
\end{align*}
For any $t>0$,
\begin{align*}
    &\p\left(\int_0^t\int_{0<|z|< n} \eta(z) N(\rd s,\rd z)=0\right)\\
    &=\lim_{k\to\infty}\p\left(\int_0^t\int_{\eps_k<|z|< n} \eta(z) N(\rd s,\rd z)=0\right)\\
    &=\lim_{k\to\infty}\p\left(\int_{\eps_k<|z|< n} \eta(z) N(t,\rd z)=0\right)\\
    &\le \lim_{k\to\infty}\p\left(N\left(t,\left(\eps_k,\frac{1}{2}\right]\right)=0\right).
\end{align*}
Here, since for any $A\in\mathcal{B}(\real)\setminus\{0\}$, $\{N(t,A)\}_{t\ge 0}$ is a Poisson process with intensity $\nu(A)$ (e.g. see \cite{Ap09} Th 2.3.5), we have
\begin{align*}
    \lim_{k\to\infty}\p\left(N\left(t,\left(\eps_k,\frac{1}{2}\right]\right)=0\right)
    &=\lim_{k\to\infty} \exp\left(-t \nu\left(\left(\eps_k,\frac{1}{2}\right]\right)\right)\\
    %&=\lim_{k\to\infty}\exp\left(-t \int_{\eps_k}^{\frac{1}{2}}\frac{c(z)}{|z|^{1+\alpha}}\rd z\right)\\
    &=0.
\end{align*}
The last equation follows from assumption \eqref{ass_nu2}.
We set for each $t>0$,
\begin{align*}
    I_t&=\int_0^t \int_{0<|z|< n} \eta (z) N(\rd s,\rd z),
\end{align*}
from countable additivity, we have
\begin{align*}
    \p\left(\bigcup_{t\in (0,\infty)\cap \mathbb{Q}} \left\{I_t=0 \right\}\right)
    \le
    \sum_{t\in (0,\infty)\cap \mathbb{Q}}\p\left( \left\{I_t=0 \right\}\right)
    =0 \quad(\textrm{see e.g. \cite{Wi12} 1.9(b)}).
\end{align*}
Thus we obtain
\begin{align*}
    \p\left(\bigcap_{t\in (0,\infty)\cap \mathbb{Q}} \left\{I_t \neq 0 \right\}\right)=1.
\end{align*}
Here, since $\eta \ge 0$, we obtain
\begin{align*}
    \p&\left(\bigcap_{0\le s\le u}\left\{I_s\le I_u\right\}\right)=1.
\end{align*}
By tightness of rational numbers, we obtain the following:
\begin{align*}
    \p\left(\forall t>0,\ I_t\neq 0\right)
    &\ge \p\left(\bigcap_{t\in (0,\infty)\cap \mathbb{Q}}\left\{ I_t\neq 0 \right\}\cap \bigcap_{0\le s\le u}\left\{I_s\le I_u\right\}\right)\\
    &=1.
\end{align*}

\subsection{Preparation for proof of convergence of $X^{(n)}$}\label{Preparation}
%Before proving Theorem $\ref{mainthm}$, we prove some lemmas which are used in the evaluation of $\left|X_t-\tilde{X}_t\right|^p$ in the method introduced by Komatsu (\cite{Ko82}, proof of Theorem \ref{mainthm}).
To prove Theorem \ref{mainThm1}, we apply a variation of the method introduced by Komatsu (\cite{Ko82}, proof of Theorem 1) in order to prove to converge of $X^{(n)}$.
This technique has been used to \cite{Na20}.
\begin{Lem}\label{abstruct}
For $\eps >0$, $\delta>1$ and $r\in(0,1]$, we can choose a smooth function $\pde$ which satisfies the following conditions,
%\begin{align*}
%\mbox{supp}\pde \subset (-\eps,-\eps \delta^{-1}) \cup (\eps \delta^{-1},\eps),\\
%0\le \pde(x) \le  2(x \log \delta)^{-1} \textrm{ for }\eps \delta^{-1}<|x|<\eps,\\
%\end{align*}
\begin{equation*}
\pde(x) = \begin{cases}
	\textrm{between} \ 0\textrm{ and }2(x \log \delta)^{-1} &\quad \eps \delta^{-1}<x<\eps,\\
   0  &\quad \textrm{otherwise},
  \end{cases}
\end{equation*}
and $\int_{\eps \delta^{-1}}^{\eps} \pde(y) dy=1$.
We define $u_r(x)=|x|^r$ and $\urde=u_r\ast \pde$.
Then, $\urde \in C^2$ and for any $x\in \real$,
\begin{align}
\label{ude_ine1}|x|^r &\le \eps^r+\urde(x),\\%,\ x \in \real.
\label{ude_ine2}\urde(x)&\le |x|^r+\eps^r.
\end{align}
\end{Lem}
We introduce a quasi-martingale and its properties.
Let $T\in[0,\infty]$ %and $\left\{\mf_t\right\}_{t\ge 0}$ be  a filtration satisfying the  usual conditions 
and $Z$ be a c\`adl\`ag adapted process defined on $[0,T]$.
A finite subdivision of $[0,T]$ is defined by
$\Delta t=(t_0,t_1,\ldots,t_{n+1})$ such that $0=t_0<t_1<\cdots<t_{n+1}=T$.
\begin{Def}
The mean variation of $X$ is defined by
\begin{equation*}
V_T(X):=\sup_{\Delta t}\e\left[\sum_{i=0}^n\left|\e\left[X_{t_i}-X_{t_{i+1}}|\mf_{t_i}\right]\right|\right].
\end{equation*}
\end{Def}
\begin{Def}
A c\`adl\`ag adapted process $Z$ is a quasi-martingale on $[0,T]$ if for each $t\in[0,T]$, $\e[|Z_t|]<\infty$ and $V_T(Z)<\infty$.
\end{Def}
%Also recall that by convention if $X$ is defined only on $[0,\infty)$, then we set $X_\infty=0$.
%The following lemma is proved by Kurtz \cite{Ku91} using Rao's theorem (\cite{Philip}, Section III, Theorem 17).
Kurtz \cite{Ku91} proved the following lemma by using Rao's theorem (\cite{Pr05}, Section III, Theorem 17).
% is used in the proof.
%\begin{Lem}(\cite{Ku91}, Lemma 5.1)
%Suppose that $T\in[0,\infty), t\ge 0$, $\e[|X_t|]<\infty$ and $V_t(X)<\infty$. Set a stopping time $\tau$ and $\Delta \tau=(\tau_0,\tau_1,\ldots,\tau_{n+1})$:stopping times $0\le \tau_0 \le\ldots \le \tau_{n+1}\le T$. Then,
%\begin{equation*}
%V_T(X)=\sup_{\Delta \tau}\e\left[\sum_{i=0}^n\left|\e[X_{\tau_i}-X_{\tau_{i+1}}\bigm|\mf_{\tau_i}]\right|\right] \textrm{~~and~~}
%\e\left[\left|X_{\tau\wedge T}\right|\right]\le V_T(X)+\e\left[\left|X_T\right|\right].
%\end{equation*}
%\end{Lem}Using the above lemma, the next lemma is proved.
\begin{Lem}(\cite{Ku91}, Lemma 5.3)\label{quasiprop}
Let $Z$ be a c\`adl\`ag adapted process defined on $[0,T]$. Suppose that for each $t\in[0,T]$, $\e[|Z_t|]<\infty$ and $V_t(Z)<\infty$. Then, for each $h>0$,
\begin{equation*}
h\p\left(\sup_{t\in[0,T]}\left|Z_t\right|>h\right)\le V_T(Z)+\e\left[\left|Z_T\right|\right].
\end{equation*}
\end{Lem}
\subsection{Proof of $L^p$-convergence of $X^{(n)}$}\label{ProofofLpconv}
In this section, we prove several lemmas used in proving Lemma \ref{DthetaXt} and satisfying the assumptions of Lemma \ref{SoXieThm32}.
The key point in this proof is to use the fact that $\{X_t^{(\ast,n)}\}_{t\ge 0}$ converges towards the law of $\{X_t\}_{t\ge 0}$ in $L^p$, where $p\in(1,\beta)$.\\
To that purpose, we show the following two statements.
\begin{Lem}\label{XnX_Lp_conv}
\begin{itemize}
    \item [(i)] $\sup_{t\in[0,T]}|X_t^{(n)}-X_t|^p\to 0$ in probability as $n\to\infty$.
    \item [(ii)] The class of random variable
    \begin{align*}
    \left\{\sup_{s\in[0,t]}\left|X_s^{(n)}-X_s\right|^p\right\}_{t\in[0,T]}
    \end{align*}
    is uniformly integrable.
\end{itemize}    
\end{Lem}
\begin{proof}
    (i):
    We write $r=\frac{p}{\beta}$ for clarity.
By using the triangle inequality and Jensen's inequality, we have
\begin{align*}
    \left|X_t^{(n)}-X_t\right| 
    &\le \int_0^t \left|b\left(X_{s}^{(n)}\right)-b\left(X_{s}\right)\right|\rd s + \sigma_2 \left|L_t^{(n)}-L_t\right|.
\end{align*}
By the definition of supremum, we have
\begin{align*}
    \sup_{t\in[0,T]}\left|X_t^{(n)}-X_t\right|
    &\le T \sup_{t\in[0,T]}\left|b\left(X_{t}^{(n)}\right)-b\left(X_{t}\right)\right| + \sigma_2 \sup_{t\in[0,T]}\left|L_t^{(n)}-L_t\right|.
\end{align*}
Since $b$ is Lipschitz continuous, we have
\begin{align*}
    \sup_{t\in[0,T]}\left|X_t^{(n)}-X_t\right|
    &\le  CT\sup_{t\in[0,T]}\left|X_{t}^{(n)}-X_{t}\right|\rd s + \sigma_2 \sup_{t\in[0,T]}\left|L_t^{(n)}-L_t\right|.
\end{align*}
By using Gronwall's inequality and Jensen's inequality, we have
\begin{align}\label{XnXLnL}
    \sup_{t\in[0,T]}\left|X_t^{(n)}-X_t\right|
    \le C\sigma_2 \exp(CT) \sup_{t\in[0,T]}\left|L_t^{(n)}-L_t\right|.
\end{align}
Here, by above inequality and Lemma \ref{quasiprop}, for any $h>0$
\begin{align*}
 \p\left(\sup_{t\in[0,T]}\left|X_t^{(n)}-X_t\right|^p> h\right)
    &\le  \p\left(\sup_{t\in[0,T]}\left|L_t^{(n)}-L_t\right|^{r}> \left(\frac{h}{C \sigma_2 \exp(CT)}\right)^{\frac{1}{\beta}}\right)\\
    &\le  \p\left(\sup_{t\in[0,T]}\left(\eps^p+u_{r,\delta,\varepsilon}\left(L_t^{(n)}-L_t\right) \right)>\left(\frac{h}{C \sigma_2 \exp(CT)}\right)^{\frac{1}{\beta}}\right)\\
    &\le \left(\frac{C \sigma_2 \exp(CT)}{h}\right)^{\frac{1}{\beta}} \bigg(V_T\left(\eps^p+u_{r,\delta,\varepsilon}\left(L^{(n)}-L\right)\right)\\
    &\qquad+\e\left[\left|\eps^p+u_{r,\delta,\varepsilon}\left(L_T^{(n)}-L_T\right)\right|\right]\bigg).
\end{align*}
Here, by the definition of the mean variation and \eqref{ude_ine2}, we have
\begin{align*}
V_T\left(\eps^p+u_{r,\delta,\varepsilon}\left(L^{(n)}-L\right)\right)&=V_T\left(u_{r,\delta,\varepsilon}\left(L^{(n)}-L\right)\right),\\
\e\left[\left|\eps^p+u_{r,\delta,\varepsilon}\left(L_t^{(n)}-L_t\right)\right|\right]&=\e\left[\eps^p+u_{r,\delta,\varepsilon}\left(L_T^{(n)}-L_T\right)\right].
\end{align*}
By using the L\'evy-It\^o decomposition (\cite{Ap09}, Theorem 2.4.16), we have
\begin{align*}
L_t^{(n)}-L_t
&=\int_{0}^{t}\int_{|z|\ge 1} z\1_{\{0<|z|< n\}} N(\rd s,\rd z)+\int_{0}^{t}\int_{0<|z|< 1}z\1_{\{0<|z|< n \}}\nt(\rd s,\rd z)\\
&\qquad -\int_{0}^{t}\int_{|z|\ge 1} z N(\rd s,\rd z)-\int_{0}^{t}\int_{0<|z|< 1} z\nt(\rd s,\rd z),\\
&=-\int_{0}^{t}\int_{|z|\ge 1} z\1_{\{|z|\ge n\}} N(\rd s,\rd z)-\int_{0}^{t}\int_{0<|z|< 1}z\1_{\{|z|\ge n \}}\nt(\rd s,\rd z),\\
&=-\int_{0}^{t}\int_{|z|\ge 1} z\1_{\{|z|\ge n\}} N(\rd s,\rd z).
\end{align*}
%For the function $\urde$ in Lemma \ref{abstruct}, using It\^o's formula (\cite{Ap09}, Theorem 4.4.7), we observe
The last equal follows by $n\in\n$.
Using the It\^o's formula (\cite{Ap09}, Theorem 4.4.7) and $N(\rd t,\rd z)=\nt(\rd t,\rd z)+\nu(\rd z)\rd t$, we have
\begin{align*}
    &\urde \left(L_t^{(n)}-L_t\right)\\
    &=\int_{0}^{t}\int_{|z|\ge 1} \left\{\urde\left(L_{s-}^{(n)}-L_{s-}-z\1_{\{|z|\ge n\}}\right)-\urde\left(L_{s-}^{(n)}-L_{s-}\right)\right\} N(\rd s,\rd z)\\
    &=\int_{0}^{t}\int_{|z|\ge 1} \left\{\urde\left(L_{s-}^{(n)}-L_{s-}-z\1_{\{|z|\ge n\}}\right)-\urde\left(L_{s-}^{(n)}-L_{s-}\right)\right\} \nt(\rd s,\rd z)\\
    &\qquad +\int_{0}^{t}\int_{|z|\ge 1} \left\{\urde\left(L_{s-}^{(n)}-L_{s-}-z\1_{\{|z|\ge n\}}\right)-\urde\left(L_{s-}^{(n)}-L_{s-}\right)\right\} \rd s\nu(\rd z),\\
    &=: M_t^{\delta,\eps}+I_t^{\delta,\eps}.
\end{align*}
Here, by \eqref{ude_ine1}, for any $x,y\in\real$
\begin{align*}
    -\urde(y)
        &\le \eps^r - |y|^r,\\
    \urde(x)-\urde(y)
        &\le 2\eps^r +|x|^r- |y|^r,\\
        &\le 2\eps^r +\left| |x|^r- |y|^r \right|,\\
        &\le 2\eps^r +|x-y|^r.
\end{align*}
So we have,
\begin{align*}
    &\urde \left(L_T^{(n)}-L_T\right)\\
    &\le \int_{0}^{T}\int_{|z|\ge 1} \left\{\urde\left(L_{s-}^{(n)}-L_{s-}-z\1_{\{|z|\ge n\}}\right)-\urde\left(L_{s-}^{(n)}-L_{s-}\right)\right\} \nt(\rd s,\rd z)\\
    &\qquad +\int_{0}^{T}\int_{|z|\ge 1} \left\{2 \eps^r+|z|^r\1_{\{|z|\ge n\}}\right\} \rd s\nu(\rd z).
\end{align*}
Also,
\begin{align*}
    \int_{0}^{T}\int_{|z|\ge 1} \left\{2 \eps^r+|z|^r\1_{\{|z|\ge n\}}\right\} \rd s\nu(\rd z)
    \le 2 C T \eps^r+T \int_{\real\setminus\{0\}} |z|^r \1_{\{|z|\ge n\}} \nu(\rd z).
\end{align*}
We can evaluate
\begin{align*}
    &V_T\left(u_{r,\delta,\varepsilon}\left(L^{(n)}-L\right)\right)\\
    &=\sup_{\Delta t}\e\left[\sum_{i=0}^n \left|\e\left[ u_{r,\delta,\varepsilon}\left(L^{(n)}_{t_{i}}-L_{t_{i}}\right)-u_{r,\delta,\varepsilon}\left(L^{(n)}_{t_{i+1}}-L_{t_{i+1}}\right)\mid \mathcal{F}_{t_i} \right]\right|\right]\\
    &=\sup_{\Delta t}\e\left[\sum_{i=0}^n \left|\e\left[ M_{t_i}^{\delta,\eps}+I_{t_i}^{\delta,\eps}-M_{t_{i+1}}^{\delta,\eps}-I_{t_{i+1}}^{\delta,\eps}\mid \mathcal{F}_{t_i} \right]\right|\right]\\
    &=\sup_{\Delta t}\e\left[\sum_{i=0}^n \left|\e\left[ I_{t_i}^{\delta,\eps}-I_{t_{i+1}}^{\delta,\eps}\mid \mathcal{F}_{t_i} \right]\right|\right].
\end{align*}
Since $(M_t^{\delta,\eps})_{t\in[0,T]}$ is a martingale, the last equation follows.
By using Jensen's inequality, we have
\begin{align*}
    &V_T\left(u_{r,\delta,\varepsilon}\left(L^{(n)}-L\right)\right)\\
    &\le \sup_{\Delta t}\e\left[\sum_{i=0}^n \e\left[ \left|I_{t_i}^{\delta,\eps}-I_{t_{i+1}}^{\delta,\eps}\right|\mid \mathcal{F}_{t_i} \right]\right]\\
    &= \sup_{\Delta t}\sum_{i=0}^n \e\left[ \left|\int_{t_i}^{t_{i+1}} \int_{|z|\ge 1} \left\{\urde\left(L_{s-}^{(n)}-L_{s-}-z\1_{\{|z|\ge n\}}\right)-\urde\left(L_{s-}^{(n)}-L_{s-}\right)\right\} \rd s\nu(\rd z)\right| \right]\\
    &\le \sup_{\Delta t}\sum_{i=0}^n \e\left[ \int_{t_i}^{t_{i+1}} \int_{|z|\ge 1} \left|\urde\left(L_{s-}^{(n)}-L_{s-}-z\1_{\{|z|\ge n\}}\right)-\urde\left(L_{s-}^{(n)}-L_{s-}\right)\right| \rd s\nu(\rd z) \right]\\
    &\le \sup_{\Delta t}\sum_{i=0}^n \e\left[ \int_{t_i}^{t_{i+1}} \int_{|z|\ge 1} \left(2 \eps^r+|z|^r\1_{\{|z|\ge n\}}\right) \rd s\nu(\rd z) \right]\\
    &\le \int_{0}^{T} \int_{|z|\ge 1} \left(2 \eps^r+|z|^r\1_{\{|z|\ge n\}}\right) \rd s\nu(\rd z).
\end{align*}
Therefore, we have
\begin{align*}
     &\p\left(\sup_{t\in[0,T]}\left|X_t^{(n)}-X_t\right|^p> h\right)\\
     &\le \left(\frac{C \sigma_2 \exp(CT)}{h}\right)^{\frac{1}{\beta}} \left\{(4 C T+1)\eps^r+2T \int_{\real\setminus\{0\}} |z|^r \1_{\{|z|\ge n\}} \nu(\rd z)\right\}.
\end{align*}
Since the above inequality holds for any $\eps>0$, we obtain
\begin{align*}
    \p\left(\sup_{t\in[0,T]}\left|X_t^{(n)}-X_t\right|^p> h\right)\le 2T\left(\frac{C \sigma_2 \exp(CT)}{h}\right)^{\frac{1}{\beta}} \int_{\real\setminus\{0\}} |z|^r \1_{\{|z|\ge n\}} \nu(\rd z).
\end{align*}
By the assumption, for any $h>0$, we have 
\begin{align*}
    \lim_{n\to \infty} \p\left(\sup_{t\in[0,T]}\left|X_t^{(n)}-X_t\right|^p> h\right)=0.
\end{align*}
Here,
\begin{align*}
    \left|\sup_{t\in[0,T]}X_t^{(n)}-\sup_{t\in[0,T]}X_t\right|
    \le \sup_{t\in[0,T]}\left|X_t^{(n)}-X_t\right|.
\end{align*}
Therefore, we have
\begin{align*}
    \lim_{n\to \infty} \p\left(\left|\sup_{t\in[0,T]}X_t^{(n)}-\sup_{t\in[0,T]}X_t\right|^p> h\right)=0.
\end{align*}
(ii): Next, we show that a following process
\begin{align*}
    \left\{\sup_{t\in [0,T]} \left|X_t^{(n)}-X_t\right|^p \right\}_{T\ge 0}
\end{align*}
is uniform integrability.
 To show that, by using inequality \eqref{XnXLnL} it suffices to show for some $q>1$
 \begin{align*}
     \e\left[\left(\sup_{t\in[0,T]}\left|L_t^{(n)}-L_t\right|^p\vee 1\right)^q\right]<\infty.
 \end{align*}
By assumption of \eqref{ass_nu1} and the denseness in real numbers, we can set $q>1$ such that $pq<\beta$ so that for each $n\in\n$.
\begin{align}\label{pq_nu_ine}
    \int_{|z|\ge n} |z|^{pq}\nu(\rd z)<\infty.
\end{align}
Here, we set $g(x)=|x|^{pq} \vee 1$, then $g$ is a nonnegative increasing submultiplicative function and $\lim_{x\to\infty} g(x)=\infty$.
Since Theorem 25.18 in \cite{Sa99}, we should show the following:
\begin{align*}
    \e\left[g\left(\left|L_t^{(n)}-L_t\right|\right)\right]<\infty\quad \textrm{~~for~~some~~}t>0.
\end{align*}
For some $t>0$,
\begin{align*}
    \e\left[g\left(\left|L_t^{(n)}-L_t\right|\right)\right]
    &=\e\left[\1_{\left\{\left|L_t^{(n)}-L_t\right|\le 1\right\}}\right]+
    \e\left[\left|L_t^{(n)}-L_t\right|^{pq} \1_{\left\{\left|L_t^{(n)}-L_t\right|>1\right\}}\right]\\
    &\le 1+\e\left[\left|L_t^{(n)}-L_t\right|^{pq}\right]\\
    %&\le 1+\eps^{pq} +\e\left[u_{r,\delta,\varepsilon}^{(q)}\left(L_t^{(n)}-L_t\right)\right]\\
    &<\infty,
\end{align*}
%where $u_{r,\delta,\varepsilon}^{(q)}:=u^q \ast \psi_{\delta,\eps}$.
The last inequality does not depend on $n\in\n$ because of \eqref{pq_nu_ine} and Example 25.10 in \cite{Sa99}.
%The last inequality is obtained in the same way as in (i).
\end{proof}
\begin{Lem}\label{lem_CnC_Lp_conv}
    We show 
    \begin{align*}
        \lim_{n\to\infty} \e\left[\sup_{t\in[0,T]}\left|C_t^{(n)}-C_t\right|^p\right]=0.
    \end{align*}
\end{Lem}
\begin{proof}
    It can be proved in the same way as in Lemma \ref{XnX_Lp_conv}.    
\end{proof}
\end{appendices}

\end{document}